\documentclass[reqno]{amsart}
\usepackage{amssymb,amsmath,amsthm,amsfonts,fancyhdr}
\usepackage{indentfirst}
\usepackage{mathrsfs}
\usepackage{setspace}
\usepackage{color}
\usepackage{cite}
\usepackage{bm}
\usepackage{lineno}
\usepackage{latexsym}
\usepackage{verbatim}
\usepackage{cases}
\usepackage{esint}
\usepackage{enumerate}
\numberwithin{equation}{section}
\newtheorem{thm}{Theorem}[section]

\newtheorem{lem}[thm]{Lemma}

\newtheorem{rem}[thm]{Remark}

\newtheorem{conj}[thm]{Conjecture}

\newcommand{\re}{\text{Re }}

\numberwithin{equation}{section}
\numberwithin{figure}{section}
\allowdisplaybreaks[3]

\begin{document}
	
	\title[A refinement of Pawlowski's result]{A refinement of Pawlowski's result}

	\author{Teng Zhang}

	\address{School of Mathematics and Statistics, Xi'an Jiaotong University, Xi'an, P.R.China 710049.}
	
	\email{teng.zhang@stu.xjtu.edu.cn}

	\begin{abstract}
Let \(F(z) = \prod_{k=1}^{n}(z - z_k)\) be a monic complex polynomial of degree \(n\) whose zeros satisfy \(\max\limits_{1 \le k \le n} |z_k| \le 1\). Pawłowski [Trans. Amer. Math. Soc. 350(11) (1998)] considered the radius \(\gamma_n\) of the smallest disk, centered at the centroid \(\frac{1}{n}\sum_{k=1}^n z_k\), containing at least one critical point of \(F\), establishing the bound $\gamma_n \le \frac{2\,n^{\frac{1}{n-1}}}{n^{\frac{2}{n-1}} + 1}$. In this paper, inspired by the spirit of Borcea's variance conjectures and leveraging the classical Schoenberg inequality, we significantly refine Pawłowski's estimate by proving succinctly and elegantly that
$\gamma_n \le \sqrt{\frac{n - 2}{n - 1}}$. This result also represents a rare and noteworthy application of Schoenberg's inequality to the geometry of polynomial critical points.
	\end{abstract}
	
	
	\maketitle
	
	\section{Introduction}
	In 1958, Sendov proposed his well-known conjecture concerning the relative locations of the zeros and critical points of complex polynomials. This problem was later featured in Hayman's influential research problems book \cite{Hay19}.
	\begin{conj}[Sendov Conjecture] Let $ F $ be a complex polynomial of degree $ n \geq 2 $ with zeros $ z_1, \dots, z_n $ (all satisfying $ |z_k| \leq 1 $) and critical points $ w_1, \dots, w_{n-1} $. Then
		$$ \max_{1 \le k \le n} \, \min_{1 \le j \le n-1} |z_k - w_j| \le 1. $$
	\end{conj}
Despite numerous ingenious approaches developed in attempts to resolve Sendov’s conjecture, only partial results are currently known (see, e.g., \cite{BX99,Tao22}). Among contributors to this problem, Schmeisser \cite{Sch77} proposed the following generalization of the conjecture.
	\begin{conj}[Schmeisser Conjecture]Let $ F $ be a complex polynomial of degree $ n \geq 2 $ with zeros $ z_1, \dots, z_n $ (all satisfying $ |z_k| \leq 1 $) and critical points $ w_1, \dots, w_{n-1} $. Then for  any nonnegative weights $l_1, \ldots, l_n\ge 0$ satisfying $\sum_{k=1}^nl_k=1$, we have
		\[
		\min_{1\le j\le n-1}\left|\sum_{k=1}^nl_k z_k-w_j\right|\le 1.
		\]
	\end{conj} 
The classical Gauss-Lucas Theorem \cite[p. 18]{BE95} states that for any non-constant complex polynomial $F$, its critical points lie in the convex hull of the zeros of $F$. Building on this theorem, Pawlowski \cite{Paw98} noted that the Gauss-Lucas Theorem exactly establishes the barycenter special case of the Schmeisser Conjecture. However, Pawlowski did not supply an independent proof of this special case. To prevent potential confusion among readers, we provide one in this paper.
	\begin{thm}\label{1.1t}
	Let $ F $ be a complex polynomial of degree $ n \geq 2 $ with zeros $ z_1, \dots, z_n $ (all satisfying $ |z_k| \leq 1 $) and critical points $ w_1, \dots, w_{n-1} $. Then
		\begin{eqnarray}\label{1.1e}
			\min_{1\le j\le n-1}\left|\dfrac{\sum_{k=1}^n z_k}{n}-w_j\right|\le 1.
		\end{eqnarray}
	\end{thm}
	\begin{proof}
		It is sufficient to show
		\begin{eqnarray}\label{e1.1}
			\sum_{j=1}^{n-1}\left|\dfrac{\sum_{k=1}^n z_k}{n}-w_j\right|^2\le n-1.
		\end{eqnarray}
		Compute
		\begin{eqnarray*}
			\text{LHS}_{(\ref{e1.1})}&=& \sum_{j=1}^{n-1}\left|\dfrac{\sum_{k=1}^n z_k}{n}\right|^2-2\sum_{j=1}^{n-1}\overline{w_j} \dfrac{\sum_{k=1}^{n} z_k}{n}+\sum_{j=1}^{n-1}\left|w_j\right|^2\\
			&=&\sum_{j=1}^{n-1}\left|w_j\right|^2-(n-1)\left|\dfrac{\sum_{k=1}^n z_k}{n}\right|^2 \\
			&&\text{(Barycenter invariance gives  that $\tfrac{\sum_{j=1}^{n-1}w_j}{n-1}=\tfrac{\sum_{k=1}^{n}z_k}{n}$)}\\
			&\le&\sum_{j=1}^{n-1}\left|w_j\right|^2\\
			&\le& n-1. \\
			&& \text{(Since all $\left|z_k\right|\le 1$, by Gauss-Lucas Theorem)}
		\end{eqnarray*}
	\end{proof}
Pawlowski's result \cite{Paw98} shows that the radius 1 on the right-hand side of (\ref{1.1e}) can be replaced by a smaller radius which depends only on the degree of $F$.
	\begin{thm}\label{pawt}
Let $ F $ be a complex polynomial of degree $ n \geq 2 $ with zeros $ z_1, \dots, z_n $ (all satisfying $ |z_k| \leq 1 $) and critical points $ w_1, \dots, w_{n-1} $.  Assume that $\gamma_n$ is the radius of the smallest concentric disk centered at $\tfrac{1}{n} \sum_{k=1}^n z_k$ containing at least one critical point of $F$. Then
$$ n^{-\frac{1}{n-1}} \le \gamma_n \le \frac{2 n^{\frac{1}{n-1}}} {n^{\frac{2}{n-1}} + 1}. $$
Moreover, $\gamma_n = \frac{1}{3}$ if  $n=3$ or  all zeros of $F(z)$ are real.
	\end{thm}
We note that the bound in \cite[Theorem 3.3]{Paw98} contains a typographical error: the correct inequality is $\gamma_n \le \tfrac{2n^{\frac{1}{n-1}}}{n^{\frac{2}{n-1}} + 1}$ rather than $\tfrac{2n^{\frac{1}{n-1}}}{n^{\frac{2}{n+1}} + 1}$. Pawlowski \cite{Paw98} observed that the sharp lower bound $n^{-\frac{1}{n-1}} \le \gamma_n$ follows from the special case $p(z) = z^n - z$ for $n\ge 3$. Specifically, Vieta's formulas show that the barycenter of the zeros of $p(z)$ is at 0. Furthermore, $p'(z) = nz^{n-1} - 1$, and each critical point $\zeta_k$ has modulus $n^{-\frac{1}{n-1}}$. However, the upper bound $\gamma_n \le \tfrac{2n^{\frac{1}{n-1}}}{n^{\frac{2}{n-1}} + 1}$ does not appear to be sharp, as $n^{-\frac{1}{n-1}}$ is  asymptotically equivalent to $1 - \tfrac{\log n}{n}$ but
	\begin{eqnarray*}
	 \tfrac{2n^\frac{1}{n-1}}{n^\frac{2}{n-1}+1}&=&1-\dfrac{(n^{\frac{1}{n-1}}-1)^2}{n^\frac{2}{n-1}+1}\\
		&=&1-\dfrac{(\exp(\tfrac{\log n}{n-1})-1)^2}{n^\frac{2}{n-1}+1}\\
		&\sim&1-\dfrac{\left( \tfrac{\log n}{n-1}\right) ^2}{2}\\
		&\sim&1-\dfrac{1}{2}\left( \dfrac{\log n}{n}\right)^2,
	\end{eqnarray*}
where $\sim$ denotes asymptotic equivalence as $n \to \infty$. Based on these observations, Pawlowski \cite{Paw98} conjectured that
	\begin{conj}\label{pawc}
	A sharp upper bound for $\gamma_n$ should be asymptotically of the form $\gamma_n \le 1 - c \frac{\log n}{n}$ for some positive constant $c$.
	\end{conj}
	
We focus on this interesting problem in the present paper. In the next section, we refine the upper bound for $\gamma_n$ using techniques inspired by the Borcea variance conjectures and Schoenberg's inequality.
	\section{Main results}\label{s-2}
	
Schoenberg's conjecture \cite{Scho86} proposes a polynomial analogue of Rolle's theorem relating the zeros and critical points of polynomials.
\begin{conj}[Schoenberg Conjecture] \label{SchC} Let $ F $ be a complex polynomial of degree $ n \geq 2 $ with zeros $ z_1, \dots, z_n $  and critical points $ w_1, \dots, w_{n-1} $. Then
		\begin{eqnarray}\label{schi}
		\sum_{j=1}^{n-1}\left|w_j\right|^2\le \dfrac{1}{n^2}\left|\sum_{k=1}^n z_k\right|^2+\dfrac{n-2}{n}\sum_{k=1}^n\left|z_k\right|^2,
	\end{eqnarray}
where equality holds iff  all $z_1, \ldots, z_n$ are collinear in the complex plane.
\end{conj}

Conjecture \ref{schi} was confirmed independently by Pereira \cite{Per03} and Malamud \cite{Mal04}. It was subsequently reproved by Cheung–Ng \cite{CN06} and Kushel–Tyaglov \cite{KT16} using distinct approaches. For studies on fourth-order Schoenberg inequalities, see \cite{BS97,CN06,KT16}.

Let $ F $ be a complex polynomial of degree $ n \geq 2 $ with zeros $ z_1, \dots, z_n $. Define the $p$-variance of the zeros of $F$ by
$$ \sigma_p(F) = \min_{c \in \mathbb{C}} \left( \frac{1}{n} \sum_{k=1}^n |z_k - c|^p \right)^{\frac{1}{p}}. $$
We focus on particular cases: $\sigma_{1}(F)$ is the mean deviation, $\sigma_2(F)$ is the variance, and we define $\sigma_\infty(F)$ as
$$ \sigma_{\infty}(F) = \lim_{p\to\infty} \sigma_p(F) = \min_{c \in \mathbb{C}} \max_{1 \le k \le n} |z_k - c|. $$
Khavinson et al. \cite[Lemma 2.1]{KPPSS11} proved that $\sigma_{p}(F)$ is non-decreasing in $p$ for any fixed polynomial $F(z)$.
	\begin{lem}\label{kle} For every polynomial $F$ and $0<p<q<\infty$, we have
		\[
		\sigma_p(F)\le \sigma_q(F)\le \sigma_\infty(F).
		\]	
	\end{lem}
While Khavinson et al.'s proof uses Hölder's inequality, we establish the result using only the convexity of the function $t\mapsto t^\alpha$ for $ \alpha\ge 1$.\\
	\noindent\emph{Alternative proof of Lemma \ref{kle}. }For any fixed c,  compute
	\begin{eqnarray*}
		\left( \dfrac{1}{n}\sum_{k=1}^n\left|z_k-c\right|^p\right)^\frac{1}{p}&=&\left( \left( \dfrac{1}{n}\sum_{k=1}^n\left|z_k-c\right|^p\right)^\frac{q}{p}\right)^\frac{1}{q}\\
		&\le&\left(  \dfrac{1}{n}\sum_{k=1}^n\left|z_k-c\right|^q\right)^\frac{1}{q}.
	\end{eqnarray*}
	Taking minimum of both sides of the above inequality yields our desired result. \qed
	
	For the variance $\sigma_2(F)$, by using elementary inner product space techniques, we have
	\begin{eqnarray}\label{s2e}
		n\sigma^2_2(F)=\min_{c\in \mathbb{C}}\sum_{k=1}^n\left|z_k-c\right|^2=\sum_{k=1}^n\left|z_k-\dfrac{z_1+\cdots+z_n}{n}\right|^2.
	\end{eqnarray}
	
	The following conjecture was derived by Borcea from his journey in Sendov conjecture \cite{Bor96, Bor96+}, which is a new viewpoint in the context of statistics of point charges in the plane.
	\begin{conj}[Borcea Variance Conjecture]Let $ F $ be a complex polynomial of degree $ n \geq 2 $ with zeros $ z_1, \dots, z_n $  and critical points $ w_1, \dots, w_{n-1} $.  Then for any $p\ge 1$, we have
		\[
		\max_{1\le k\le n}\min_{1\le j\le n-1}\left|z_k-w_j\right|\le \sigma_{p}(F).
		\]
	\end{conj}
	Now we show that for  $p=\infty$, the Borcea variance conjecture reduces to the Sendov conjecture. Define the polynomial family $\mathbb{F}=\{F(z)=\prod_{k=1}^n(z-z_k) : \max\limits_{1\le k\le n}\left|z_k\right|\le 1\}$, we  need only to show that  $\sigma_\infty:=\max\limits_{F\in \mathbb{F}}\sigma_\infty(F)=\max\limits_{F\in \mathbb{F}}\min\limits_{c\in\mathbb{C}}\max\limits_{1\le k\le n}$ $\left|z_k-c\right|=1$. Clearly, taking $c=0$, we have $\sigma_\infty\le\max\limits_{F\in \mathbb{F}}\max\limits_{1\le k\le n}\left|z_k\right|\le 1$. Furthermore, consider the polynomial $g(z)=(z+1)^{n-1}(z-1)$. For this $g$, the optimal  $c$ must be $0$, and thus $\sigma_\infty(g)=1$. This establishes $\sigma_\infty=1$.
	
	Analogous to the relationship between Schmeisser’s conjecture and Sendov’s conjecture, we propose a generalization of the Borcea variance conjecture as follows.
	\begin{conj}[Generalized  Borcea Variance Conjecture] Let $ F $ be a complex polynomial of degree $ n \geq 2 $ with zeros $ z_1, \dots, z_n $  and critical points $ w_1, \dots, w_{n-1} $. Then for any $p\ge 1$  and any nonnegative weights $l_1, \ldots, l_n\ge 0$ satisfying $\sum_{k=1}^nl_k=1$, we have
		\[
		\min_{1\le j\le n-1}\left|\sum_{k=1}^nl_k z_k-w_j\right|\le \sigma_{p}(F).
		\]
	\end{conj}
	Now, for $p\ge 2$, we confirm the barycenter case of the generalized Borcea variance conjecture using a smaller coefficient $\sqrt{\tfrac{n-2}{n-1}}$.
	\begin{thm}\label{mt}Let $ F $ be a complex polynomial of degree $ n \geq 2 $ with zeros $ z_1, \dots, z_n $  and critical points $ w_1, \dots, w_{n-1} $. Then
	\begin{eqnarray}\label{me}
				\min_{1\le j\le n-1}\left|\dfrac{\sum_{k=1}^n z_k}{n}-w_j\right|\le \sqrt{\dfrac{n-2}{n-1}}\sigma_{2}(F),
	\end{eqnarray}
where equality holds  iff $n=2$; $n=3$ with $z_1,z_2,z_3$  collinear; or $n>3$ with all zeros $z_1,\ldots,z_n$  equal.
	\end{thm}
	\begin{proof} By the explicit expression (\ref{s2e}) of $\sigma_{2}(F)$, we need to show
		\begin{eqnarray}\label{e1}
				\min_{1\le j\le n-1}\left|\dfrac{\sum_{k=1}^n z_k}{n}-w_j\right|^2\le \dfrac{n-2}{n(n-1)} \sum_{k=1}^n\left|z_k-\dfrac{\sum_{k=1}^nz_k}{n}\right|^2.
		\end{eqnarray}
			It is sufficient to show that
		\begin{eqnarray}\label{e2}
			\dfrac{1}{n-1}\sum_{j=1}^{n-1}\left|\dfrac{\sum_{k=1}^n z_k}{n}-w_j\right|^2\le \dfrac{n-2}{n(n-1)} \sum_{k=1}^n\left|z_k-\dfrac{\sum_{k=1}^nz_k}{n}\right|^2.
		\end{eqnarray}
		Equivalently,
		\begin{eqnarray}\label{2e}
			\sum_{j=1}^{n-1}\left|\dfrac{\sum_{k=1}^n z_k}{n}-w_j\right|^2\le \dfrac{n-2}{n} \sum_{k=1}^n\left|z_k-\dfrac{\sum_{k=1}^nz_k}{n}\right|^2.
		\end{eqnarray}
		Compute
		\begin{eqnarray*}
			\text{LHS}_{(\ref{2e})}&=&\sum_{j=1}^{n-1}\left|\dfrac{\sum_{k=1}^n z_k}{n}\right|^2-2\sum_{j=1}^{n-1}\re\overline{w_j} \dfrac{\sum_{k=1}^n z_k}{n}+\sum_{j=1}^n\left|w_j\right|^2\\
			&=&\sum_{j=1}^{n-1}\left|w_j\right|^2-(n-1)\left|\dfrac{\sum_{k=1}^n z_k}{n}\right|^2, \\
			&&\text{(Barycenter invariance gives that $\tfrac{\sum_{j=1}^{n-1}w_j}{n-1}=\tfrac{\sum_{k=1}^{n}z_k}{n}$)}\\
			\text{RHS}_{(\ref{2e})}&=&\dfrac{n-2}{n} \left( \sum_{k=1}^n\left|z_k\right|^2-n\left|\dfrac{ \sum_{k=1}^nz_k}{n}\right|^2\right)\\
			&=&  \dfrac{n-2}{n}\sum_{k=1}^n\left|z_k\right|^2-(n-2)\left|\dfrac{ \sum_{k=1}^nz_k}{n}\right|^2.
		\end{eqnarray*}
		Schoenberg inequality (\ref{schi}) implies that $\text{LHS}_{(\ref{2e})}\le\text{RHS}_{(\ref{2e})}$.
		
		Next, we examine the conditions for  equality  to hold in (\ref{me}). Recall that (\ref{e1}) is equivalent to (\ref{e2}) iff all critical points $w_j$ are equidistant from the barycenter. Schoenberg's inequality (\ref{schi}) holds with equality iff all zeros $z_1, \ldots, z_n$ are collinear in the complex plane. 
		
		When $n= 2$, the result is trivial.
		
		When $n= 3$,  the barycenter is $G=\frac{z_1+z_2+z_3}{3}=\frac{w_1+w_2}{2}$.  Consequently, the critical points are always equidistant from this barycenter. Furthermore, equality in \eqref{me} holds if and only if all zeros $z_1, z_2, z_3$ are collinear.

When $n\ge 4$,   we assert that equality in \eqref{me} holds if and only if all zeros $z_1, z_2, \ldots, z_n$ are identical.  Given that $z_1, \ldots, z_n$ are collinear,  we may rotate and translate the roots of $F$ to assume $z_1, \ldots, z_n \in \mathbb{R}$ (i.e., on the real axis), Let the distinct real roots  of \(F\)  be denoted by
	\[
	\alpha_1<\alpha_2<\cdots<\alpha_k,
	\]
	with multiplicities $m_1, m_2, \ldots, m_k$ respectively, where  $	\sum_{i=1}^k m_i = n$. Define the barycenter $G$ and related quantities:
	\[
	G=\frac1n\sum_{i=1}^n z_i, \quad R>0,\quad
	a=G-R,\quad b=G+R.
	\]
	We assume that every critical point of $F$ (counting multiplicity) lies in the two-point set \(\{a,b\}\). Moreover, barycenter invariance yields the critical points of $F$ at positions $a$ and $b$ have the same multiplicity.

	By Rolle's theorem and multiplicity considerations, the critical points of \(F\) can be classified into two categories:
	\begin{enumerate}
		\item Repeated roots: Each \(\alpha_i\) with \(m_i\ge2\) yields a critical point at \(z=\alpha_i\) of multiplicity \(m_i-1\);
		\item Intercritical points: Each open interval \((\alpha_i,\alpha_{i+1})\) contains exactly one simple critical point, denoted \(\beta_i\), for \(i=1,\dots,k-1\).
     \end{enumerate}
	 The total multiplicity of critical points is then	$\sum_{i=1}^k (m_i-1) + (k-1) = (n-k)+(k-1) = n-1 \;\ge3$.

We now examine three subcases based on \(k\):
	
	\medskip
	\noindent\textbf{Case 1: \(k\ge4\).}
	\begin{itemize}
		\item With \(k-1\ge3\),  the simple critical points \(\beta_1,\beta_2,\beta_3\) lie in the disjoint intervals \((\alpha_1,\alpha_2)\), \((\alpha_2,\alpha_3)\) and \((\alpha_3,\alpha_4)\), respectively.
	\end{itemize}
However, these three distinct points must all lie in \(\{a,b\}\), which is impossible.
	
	\medskip
	\noindent\textbf{Case 2: \(k=3\).} The distinct roots are \(\alpha_1<\alpha_2<\alpha_3\) and \(m_1+m_2+m_3=n\ge4\).
	\begin{itemize}
			\item There are two simple critical points \(\beta_1\in(\alpha_1,\alpha_2)\) and \(\beta_2\in(\alpha_2,\alpha_3)\).
		\item The repeated–root total multiplicity is
		\((m_1-1)+(m_2-1)+(m_3-1)=n-3\ge1\), so at least one critical point appears at some \(\alpha_i\).
	\end{itemize}
 Thus, there are at least three distinct critical–point locations:
\(\{\alpha_i\}\), \(\{\beta_1\}\), and \(\{\beta_2\}\).  They cannot all lie in the two–point set \(\{a,b\}\).	This is a contradiction.
	
	\medskip
	\noindent\textbf{Case 3: \(k=2\).} The distinct roots are \(\alpha_1<\alpha_2\) with \(m_1+m_2=n\ge4\).
	\begin{itemize}
		\item There is a simple critical point \(\beta\in(\alpha_1,\alpha_2)\).
		\item The repeated–root multiplicity is \((m_1-1)+(m_2-1)=n-2\ge2\), which must be distributed between \(\alpha_1\) and \(\alpha_2\).
		\begin{enumerate}
			\item $m_1\ge 3, m_2=1$, $F$  has another critical point $\alpha_1$ of multiplicity at least $2$. This contradicts the requirement that the critical points of $F$ at  $\beta$ and $\alpha_1$ have the same multiplicity. 
			\item $m_1\ge 2, m_2\ge 2$,  $F$  has two additional distinct distinct critical points $\alpha_1, \alpha_2$. This contradicts the requirement that every critical point of $F$ (counting multiplicity) lies in the two-point set \(\{a,b\}\).
			\item $m_1=1, m_2\ge 3$, $F$  has  another critical point $\alpha_2$ of multiplicity at least $2$. This contradicts the requirement that the critical points of $F$ at  $\beta$ and $\alpha_2$ have the same multiplicity.
		\end{enumerate}
	\end{itemize}
	
	\medskip
	\noindent
	In every subcase for \(n\ge4\),  the assumption that all $n-1$ critical points lie in the two-point set $\{a,b\}$ and that critical points at $a$ and $b$ have identical multiplicity leads to a contradiction.  Therefore, when $n\ge 4$,  one must have  $R=0$, i.e., all $w_j$ are equal,  implying all $z_j$ are identical.
	\end{proof} 
	As an application of Theorem \ref{mt}, we prove our main result in this paper.
	\begin{thm}\label{mt1}
		Let $ F $ be a complex polynomial of degree $ n \geq 2 $ with zeros $ z_1, \dots, z_n $ (all satisfying $ |z_k| \leq 1 $) and critical points $ w_1, \dots, w_{n-1} $.  Assume that $\gamma_n$ is the radius of the smallest concentric disk centered at $\tfrac{1}{n} \sum_{k=1}^n z_k$ containing at least one zero of $F'(z)$. Then	\[
		\gamma_n\le\sqrt{\tfrac{n-2}{n-1}} .
		\]
	\end{thm}
	\begin{proof} First, by Lemma \ref{kle}, under the assumption that all $\left|z_k\right|\le 1$, we know that $\sigma_2(F)\le \sigma_\infty(F)\le 1$. Together with Theorem \ref{mt}, we have $\gamma_n\le\sqrt{\tfrac{n-2}{n-1}} $.
	\end{proof}
	\begin{rem}
		Theorem \ref{mt1} is an improvement of Theorem \ref{pawt}.
	\end{rem}
	\begin{proof}
		 In fact, we  need to show that for any $n\ge 3$,
		\begin{eqnarray*}
			&&\sqrt{\dfrac{n-2}{n-1}}< \dfrac{2n^\frac{1}{n-1}}{n^\frac{2}{n-1}+1}.\\	\end{eqnarray*}This is equivalent to
		\begin{eqnarray*}
			&&\dfrac{n-2}{n-1}< \left(\dfrac{2n^\frac{1}{n-1}}{n^\frac{2}{n-1}+1} \right)^2 \\
			&\Leftrightarrow&1-\dfrac{1}{n-1}< 1-\dfrac{(n^\frac{2}{n-1}-1)^2}{(n^\frac{2}{n-1}+1)^2} \\
			&\Leftrightarrow&  \dfrac{n^\frac{2}{n-1}-1}{n^\frac{2}{n-1}+1}< \sqrt{\dfrac{1}{n-1}}.
		\end{eqnarray*}
		The last inequality can be obtained  by verifying that
		\begin{eqnarray*}
			\dfrac{n^\frac{2}{n-1}-1}{n^\frac{2}{n-1}+1}&\le& \dfrac{1}{2}\log(n^\frac{2}{n-1})\\
			&=& \dfrac{1}{n-1} \log n\\
			&<& \dfrac{1}{n-1} \sqrt{n-1}\\
			&=&\sqrt{\dfrac{1}{n-1}}.
		\end{eqnarray*}
	\end{proof}
	\begin{rem}
		The estimate  of the upper bound of $\gamma_3$ given in	Theorem $\ref{mt1}$, $\gamma_3\le \tfrac{\sqrt{2}}{2}$, is not sharp, since Theorem \ref{pawt} says $\gamma_3=\tfrac{1}{3}$. Notice that $\sqrt{\tfrac{n-2}{n-1}}=\sqrt{1-\tfrac{1}{n-1}}\sim 1-\tfrac{1}{2}\left( \tfrac{1}{n}\right)>1-c\tfrac{\log n}{n}$ for any positive constant $c$, therefore  Conjecture \ref{pawc} remains open.
	\end{rem}
	\begin{rem}
	To the best of the author's knowledge, there have been many studies on the proof and generalization of Schoenberg inequality (\ref{schi}), but few have found an interesting application of this inequality. This article provides such an application.	 
	\end{rem}

	\section*{Acknowledgments}  
The author thanks Professors M. Lin and H. Sendov for their valuable suggestions.

	
	
	
	
	\bibliographystyle{elsarticle-num}
	
	

\end{document}